\documentclass{amsart}

\usepackage{amssymb}
\usepackage{amsmath}
\usepackage{amsthm}

\newtheorem {theorem}{Theorem}
\newtheorem {proposition}[theorem]{Proposition}
\newtheorem {corollary}[theorem]{Corollary}
\newtheorem {lemma}[theorem]{Lemma}
\theoremstyle{definition}
\newtheorem{definition}{Definition}
\newtheorem{example}{Example}
\theoremstyle{remark}
\newtheorem{remark}{Remark}

\newcommand{\al}{\alpha}
\newcommand{\shift}{q}

\newcommand{\si}{\sigma}
\newcommand{\er}{\varepsilon}
\newcommand{\za}{\zeta}
\newcommand{\ff}{w}
\newcommand{\FF}{\Phi}
\newcommand{\hg}{h}

\newcommand{\RRe}{\mathrm{Re\,}}
\newcommand{\IIm}{\mathrm{Im\,}}
\newcommand{\GG}{g}

\newcommand{\ee}{n}
\newcommand{\we}{w}

\newcommand{\xx}{x^\prime}

\newcommand{\hol}{\mathcal{H}ol}

\newcommand{\dom}{\mathcal D}
\newcommand{\Dbb}{\mathbb D}

\newcommand{\cd}{{\mathbb{C}}^d}
\newcommand{\Rbb}{\mathbb R}
\newcommand{\RRq}{{\mathbb R}^q}
\newcommand{\Cbb}{\mathbb C}
\newcommand{\Nbb}{\mathbb N}
\newcommand{\Zbb}{\mathbb Z}

\newcommand{\pspc}{\mathrm{pspc}}
\newcommand{\app}{\mathrm{map}}
\newcommand{\trop}{\mathrm{trop}}
\newcommand{\har}{\mathrm{har}}
\newcommand{\bd}{B_d}
\begin{document}

\title[Approximation by proper holomorphic maps]{Approximation by proper holomorphic maps and tropical power series}

\author[E.~Abakumov]{Evgeny Abakumov}

\address{Universit\'e Paris-Est, LAMA (UMR 8050), 77454 Marne-la-Vall\'ee, France}

\email{evgueni.abakoumov@u-pem.fr}

\author[E.~Doubtsov]{Evgueni Doubtsov}

\address{St.~Petersburg Department
of V.A.~Steklov Mathematical Institute,
Fontanka 27, St.~Petersburg
191023, Russia}

\email{dubtsov@pdmi.ras.ru}

\thanks{The second author was supported by the Russian Science Foundation (grant No.~14-41-00010).}

\subjclass[2010]{Primary 30D15; Secondary 14T05, 26A12, 30H99, 32A15, 32H35, 41A58, 42A55, 46E15}

\date{}

\keywords{Entire function, radial weight, proper holomorphic map, tropical power series, essential weight,
log-convex function}

\maketitle

\begin{abstract}
Let $\we$ be an unbounded radial weight on the complex plane.
We study the following approximation problem:
find a proper holomorphic map $f: \Cbb\to \Cbb^n$
such that $|f|$ is equivalent to $\we$.
We give several characterizations of those $\we$
for which the problem is solvable.
In particular, a constructive characterization is given in terms of tropical power series.
Moreover, the following natural objects and properties are involved:
essential weights on the complex plane,
approximation by power series with positive coefficients,
approximation by the maximum of a holomorphic function modulus.
Extensions to several complex variables and approximation by harmonic maps
are also considered.
\end{abstract}

\section{Introduction}\label{s_int}

Let $\dom$ denote the complex plane $\Cbb$ or the unit disk $\Dbb$ of $\Cbb$.
For $R=1$ or $R=+\infty$, let $\we: [0, R)\to (0, +\infty)$ be \textsl{a weight function},
that is, let $\we$ be non-decreasing, continuous and unbounded.
Setting $\we(z) = \we(|z|)$ for $z\in\dom$, we extend $\we$
to \textsl{a radial weight} on $\dom$.

Given a set $X$ and functions $u, v: X\to (0,+\infty)$, we write $u\asymp v$
and we say that $u$ and $v$ are \textsl{equivalent} if
\[
C_1 u(x) \le v(x) \le C_2 u(x), \quad x\in X,
\]
for some constants $C_1, C_2 >0$.

\subsection{Approximation by proper holomorphic maps}\label{ss_app_proper}
In the present paper, we are primarily interested in
the following approximation property:

\begin{definition}\label{def_n_app}
A radial weight $\we$ on $\dom$ is called \textsl{approximable by a holomorphic map}
(in brief, $\we\in\dom_{\app}$) if there exists $n\in\Nbb$ and a holomorphic map $f: \dom\to \Cbb^n$ such that
\[
|f(z)| \asymp \we(z), \quad z\in\dom.
\]
\end{definition}

Let $\hol(\dom)$ denote the space of holomorphic functions on $\dom$.

\begin{remark}\label{r_mod}
Clearly, $\we\in\dom_{\app}$ if and only if
$\we$ is \textsl{approximable by a finite sum of moduli}, that is,
there exist $f_1, f_2,\dots, f_n \in \hol(\dom)$, $n\in\Nbb$, such that
\[
|f_1(z)| + |f_2(z)|+\cdots +|f_n(z)| \asymp \we(z), \quad z\in\dom.
\]
In what follows, we often use the above property in the place of $\we\in\dom_{\app}$.
\end{remark}

For $\dom=\Dbb$, a solution of the problem under consideration
is given in \cite{AD15}.
By definition, a function $v: [0,R)\to (0,+\infty)$ is called \textsl{log-convex}
if $\log v(t)$ is a convex function of $\log t$, $0<t<R$.
We write $\we\in\dom_{\log}$ if $\we$ is equivalent to a log-convex radial weight on $\dom$.
Here and in what follows, we freely exchange a weight function and its extension to a radial weight.

\begin{theorem}[{\cite[Theorem~1.2]{AD15}}]\label{t_blms}
Let $\we$ be a radial weight on $\Dbb$. Then the following properties are equivalent:
\begin{itemize}
  \item $\we\in\Dbb_{\app};$
  \item $\we\in\Dbb_{\log}.$
\end{itemize}
\end{theorem}

A holomorphic map $f:\dom\to\Cbb^n$ is called \textsl{proper} if the preimage of every compact set is compact.
 Since $\we$ is assumed to be unbounded, a holomorphic map $f$ with property~$f\in \dom_{\app}$
 is a proper one. In fact, this observation is one of the motivations for the present paper.
Indeed, J.~Globevnik \cite{G02Edin} proved that a proper holomorphic embedding $f: \Dbb\to\Cbb^2$
may grow arbitrarily rapidly; also, he asked whether such an embedding may grow arbitrarily slowly.
By \cite[Theorem~1.2]{AD15}, if $\we\in\Dbb_{\log}$, then the property $\we\in \Dbb_{\app}$ holds with $n=2$.
Using this result, one concludes that a proper holomorphic immersion $f: \Dbb\to\Cbb^2$
or a proper holomorphic embedding $f: \Dbb\to\Cbb^3$ may have arbitrary growth
(cf.~\cite[Corollaries~2.3 and 2.4]{ADcras}).
  See Section~\ref{ss_proper} for further details related to this issue
  when $\dom$ is $\cd$ or the unit ball of $\cd$, $d\ge 1$.

Trivial examples $\we_k(t) = 1+t^{k+\frac{1}{2}}$, $k\in\Nbb$, show that
there is no direct analog of Theorem~\ref{t_blms} for $\dom=\Cbb$.
Moreover, there is no such analog if $\we: [0, +\infty)\to (0, +\infty)$
is assumed to be \textsl{rapid}, that is,
\[
\lim_{t\to\infty} {t^{-k}}\we(t) =\infty\quad \textrm{for all\ } k\in\Nbb.
\]
Indeed, there exists a rapid radial weight $\we$ on $\Cbb$ such that $\we$ is log-convex but is not
approximable by a holomorphic map; in particular, we obtain such a weight using Example~3.3 from \cite{BBT98}.

In the present paper, we prove that the property $\we\in \Cbb_{\app}$
is still equivalent to several natural or well-known conditions
formulated below.
The key constructive property is that of being equivalent to \textsl{a log-tropical} weight function
($\we\in \dom_{\trop}$).
In fact, the differences between $\dom=\Dbb$ and $\dom=\Cbb$ are illustrated by the following observation:
while $\we\in \Dbb_{\log}$ if and only if $\we\in \Dbb_{\trop}$, the property $\we\in \Cbb_{\log}$
does not imply $\we\in \Cbb_{\trop}$
even for rapid weights $\we$.
Also, we give explicit sufficient conditions and necessary conditions
related to the growth of the rapid radial weight under consideration.

\subsection{Related approximation problems and properties}
Given a radial weight $\we$ on $\dom$, \textsl{the associated weight} $\widetilde{\we}$
is defined as
\[
\widetilde{\we}(z) = \sup\left\{|f(z)|: f\in\hol(\dom),\ |f|\le\we \textrm{\ on\ } \dom\right\},
\quad z\in\dom.
\]
The notion of associated weight naturally arises in the studies of the growth space $\mathcal{A}^\we(\dom)$
which consists of $f\in\hol(\dom)$ such that
\[
\|f\|_{\mathcal{A}^\we(\dom)} = \sup_{z\in\dom}\frac{|f(z)|}{\we(z)} < \infty.
\]
The definition of $\widetilde{\we}$ was formally introduced in \cite{BBT98} in a more general setting;
see \cite{BBT98} for basic properties of $\widetilde{\we}$.
In particular, $\widetilde{\we}$ is a radial weight,
so the associated weight function $\widetilde{\we}: [0,R)\to (0, +\infty)$ is correctly
defined.
Also, $\widetilde{\we}$ is known to be log-convex (see, for example, \cite[the discussion after Corollary~1.6]{BBT98}).

Clearly, the growth space $\mathcal{A}^{\we}(\dom)$ is equal to $\mathcal{A}^{\widetilde{\we}}(\dom)$ isometrically.
Also, many results related to $\mathcal{A}^{\we}(\dom)$ are formulated in terms of $\widetilde{\we}$, thus,
it is important to distinguish those $\we$ which are equivalent to $\widetilde{\we}$.

\begin{definition}[see \cite{BBT98}]\label{d_ess}
A weight function $\we: [0, R) \to (0, +\infty)$ is called \textsl{essential}
(in brief, $\we\in\dom_{\mathrm{ess}}$) if
\[
\widetilde{\we}(t) \asymp \we(t), \quad 0\le t<R.
\]
\end{definition}

\begin{definition}\label{d_max}
A weight function $\we: [0, R) \to (0, +\infty)$ is called \textsl{approximable
by the maximum of a holomorphic function modulus}
(in brief, $\we\in\dom_{\max}$) if there exists $f\in\hol(\dom)$ such that
\[
M_f(t) \asymp \we(t), \quad 0\le t<R,
\]
where $M_f(t) = \max\{|f(z)|: |z|=t\}$.
\end{definition}
Recall that Hadamard's three-circles theorem says that $M_f(t)$ is a log-convex function.

\begin{definition}\label{d_pspc}
We say that a weight function $\we: [0, R) \to (0, +\infty)$ is \textsl{approximable
by power series with positive coefficients}
(in brief, $\we\in\dom_{\pspc}$) if there exist $a_k \ge 0$, $k=0,1,\dots$, such that
\[
\sum_{k=0}^\infty a_k t^k \asymp \we(t), \quad 0\le t<R.
\]
\end{definition}
Conditions related to
the property $\we\in\Cbb_{\pspc}$ are of interest in weighted polynomial approximation problems
(see, for example, \cite{Koo64}, \cite{Lu89}, \cite{Mer58}) and in numerical applications.

Finally, we consider a constructive approximation property related to
basic notions of \textsl{tropical geometry}.
Recall that \textsl{a tropical polynomial} in one variable is defined as
\[
\Psi(x) = \sup_{j\in E} (b_j + j x), \quad x\in\Rbb,\ b_j\in\Rbb,\ j\in E,
\]
where $E$ is a finite subset of $\Zbb_+$.
Such polynomials are objects of tropical geometry
(see, for example, monograph \cite{IMS09book}).
Following Kiselman \cite{Ki14}, we say that $\Psi(x)$
is \textsl{a tropical power series} if $E$ is an arbitrary subset of $\Zbb_+$
and the supremum under consideration is finite for all $x$.
Here we assume that $-\infty < x < +\infty$ or $-\infty < x <0$.

Given a weight function $v: [0, R)\to (0, +\infty)$, consider its logarithmic transformation
\[
\FF(x) = \FF_v(x) = \log v(e^x), \quad -\infty< x < \log R,
\]
where $\log (+\infty) = +\infty$.
Observe that $v$ is log-convex if and only $\FF_v$ is convex.
We say that $v$ is \textsl{log-tropical} if $\FF_v$ is a tropical power series.
Any tropical power series is a convex function, hence, any log-tropical weight is log-convex.

\begin{definition}\label{d_trop}
Given a weight function $\we: [0, R) \to (0, +\infty)$,
we write $\we\in\dom_{\trop}$ if $\we$ is \textsl{equivalent to a log-tropical weight function}.

%
%
\end{definition}

While the properties $\we\in\dom_{\log}$ and $\we\in\dom_{\trop}$
are equivalent for $\dom=\Dbb$ (see, e.g., Theorem~\ref{t_disk}),
this is not the case for $\dom=\Cbb$ (see Example~\ref{exm_expl}).

\begin{remark}
Each property in Definitions~\ref{def_n_app}--\ref{d_trop} holds up to equivalence.
\end{remark}

In the present paper, we obtain the following result for $\dom=\Cbb$.

\begin{theorem}\label{t_abstract}
Let $\we: [0, +\infty) \to (0, +\infty)$ be a 
weight function.
Then the following properties are equivalent:
\begin{itemize}
   \item the radial weight $\we$ on $\Cbb$ is approximable by a holomorphic map $(\we\in\Cbb_{\app});$
  \item $\we$ is essential $(\we\in\Cbb_{\mathrm{ess}});$
  \item $\we$ is approximable by the maximum of a holomorphic function modulus $(\we\in\Cbb_{\max});$
  \item $\we$ is approximable by power series with positive coefficients $(\we\in\Cbb_{\pspc});$
  \item $\we$ is equivalent to a log-tropical weight function $ (\we\in\Cbb_{\trop}).$
\end{itemize}
\end{theorem}

\begin{remark}\label{r_T2}
Hadamard's three-circles theorem and related arguments guarantee that
each property listed in Theorem~\ref{t_abstract} implies $\we\in\Cbb_{\log}$, that is, $\we$ is equivalent to a
log-convex weight function.
\end{remark}

\begin{remark}\label{r_nonrapid_2}
Theorem~\ref{t_abstract} essentially simplifies
when $\we$ is log-convex and non-rapid, that is,
\[
\lim_{t\to\infty} {t^{-k}}\we(t) <\infty\quad\textrm{for some\ } k\in\Nbb.
\]
Clearly, a non-rapid log-convex weight $\we(t)$ is equivalent at $+\infty$ to $t^\alpha$ with $\alpha>0$.
Thus, the properties in question hold for $\we(t)$
if and only if $\we(t)\asymp 1+t^m$, $t\ge 0$, for certain $m\in\Nbb$.
Also, the property $\we\in\Cbb_{\trop}$ simplifies in this case:
the logarithmic transform $\FF_{\we}$ is equivalent to a rather specific tropical power series;
namely, $\FF_{\we}$ is equivalent to a tropical polynomial.
\end{remark}

\subsection{Explicit characterizations for regular rapid weight functions}
The equivalent conditions listed in Theorem~\ref{t_abstract} are rather abstract.
A characterization of the property $\we\in\Cbb_{\pspc}$ given by U.~Schmid \cite{Sch98}
in terms of the logarithmic transform $\FF_\we$
is also quite technical.
We consider the property $\we\in\Cbb_{\trop}$
to discuss related explicit conditions.
In view of Remark~\ref{r_T2},
we may suppose, without loss of generality, that $\we$ is
${\mathcal C}^2$-smooth and log-convex, that is, $\FF_\we$ is convex.
Also, observe that the properties under consideration depend only on the behavior
of $\FF_\we$ at $+\infty$.

\begin{theorem}\label{t_explicit}
Let $\we: [0, +\infty) \to (0, +\infty)$ be a rapid weight function.
Assume that $\we$ is log-convex and ${\mathcal C}^2$-smooth.
\begin{itemize}
  \item[(i)] If $\liminf_{x\to +\infty}\FF_\we^{\prime\prime}(x)>0$, then $\we\in\Cbb_{\trop}$.
  \item[(ii)] If $\limsup_{x\to +\infty}\FF_\we^{\prime\prime}(x)=0$, then $\we\notin\Cbb_{\trop}$.
  \end{itemize}
\end{theorem}

\begin{remark}
Under assumptions of Theorem~\ref{t_explicit}, we have $\FF_\we \in {\mathcal C}^2(\Rbb)$ and
$\FF_\we^{\prime\prime}(x)\ge 0$ for all $x\in\Rbb$.
So, if $\lim_{x\to +\infty}\FF_\we^{\prime\prime}(x)$ exists (finite or infinite),
then  either (i) or (ii) applies.
Observe that Theorem~\ref{t_explicit}(ii) is not applicable to non-rapid weight functions.
Also, log-tropical
weight functions may grow arbitrarily rapidly; however, a rapid log-tropical
weight function may grow slower than any given rapid weight function.
\end{remark}

Theorem~\ref{t_explicit} provides explicit illustrations for Theorem~\ref{t_abstract}.
\begin{example}\label{exm_expl}
Let $\alpha > 1$ and let $w_\alpha$ be a weight function such that $\we_\alpha(t) = e^{(\log t)^\alpha}$, $t> e$.
Then $w_\alpha$, $\alpha\ge 2$, has the equivalent properties listed in
Theorem~\ref{t_abstract}, and $w_\alpha$, $1<\alpha<2$, does not have the properties under consideration.
Indeed, we have $ \Phi_{w_\alpha }(x) = x^\alpha, x>1$.
Hence, Theorem~\ref{t_explicit} applies.
In particular, for $1<\alpha<2$, $w_\alpha\in\Cbb_{\log}$ but $w_\alpha\notin\Cbb_{\trop}$.
\end{example}

\subsection{Organization of the paper}
In Section~\ref{s_abs_prf}, we prove an abridged Theorem~\ref{t_abstract} without
property $\we\in\Cbb_{\app}$;
also, we obtain an analog of Theorem~\ref{t_abstract} for $\dom=\Dbb$.
The proof of Theorem~\ref{t_abstract} is finished in Section~\ref{s_induction}, where
the key technical implication $\we\in\Cbb_{\trop} \Rightarrow \we\in\Cbb_{\app}$ is obtained.
Section~\ref{s_explicit} contains a proof of Theorem~\ref{t_explicit}.
Further results are presented in Section~\ref{s_further}.
In particular, Theorem~\ref{t_abstract_d} extends Theorem~\ref{t_abstract} to $\dom=\cd$, $d\ge 1$;
also, we consider approximation by harmonic maps.
The final Section~\ref{s_applic} contains applications with an emphasis on approximation
by proper holomorphic immersions and embeddings.

\smallskip
Main results of the present paper were announced in \cite{ADcras}.

\subsection*{Acknowledgements}
The authors are grateful to Jos\'e Bonet and Jari Taskinen
for useful discussions.

\section{Proof of Theorem~\ref{t_abstract}: known and basic implications}\label{s_abs_prf}

\subsection{An abridged Theorem~\ref{t_abstract}}
Given a weight function $\we: [0,+\infty) \to (0,+\infty)$,
an analog of property $\we\in\Cbb_{\trop}$ was introduced by
P.~Erd{\"o}s and T.~K{\"o}v{\'a}ri \cite{EK56} and
U.\,Schmid \cite{Sch95}, \cite{Sch98} to investigate approximations by power series with positive coefficients.
Namely, put
\[
P_\we(t) = \max\left\{\frac{t^k}{u_k}:\ k=0,1,\dots  \right\},
\]
where
\[
u_k = \sup\left\{\frac{t^k}{\we(t)}:\ 0\le t< +\infty \right\}, \quad k=0,1,\dots.
\]
In other words, one considers the pointwise maximum of the monomials $y_k(t) = a_k t^k$
such that $y_k(t) \le \we(t)$ and $y_k(t)$ reaches $\we(t)$ from below.
Clearly, $P_\we(t) \le \we(t)$.
So, the reverse inequality is of interest.

\begin{definition}\label{d_mon}
We say that a weight function $\we: [0, +\infty) \to (0, +\infty)$ is \textsl{approximable from below
by monomials} (in brief, $\we\in\dom_{\mathrm{mon}}$) if
\[
\we(t) \le C P_\we(t), \quad 0\le t< +\infty,
\]
for a constant $C > 1$.
\end{definition}

A direct inspection shows that the property $\we\in\Cbb_{\mathrm{mon}}$
is an equivalent reformulation of $\we\in\Cbb_{\trop}$.
So, 
we freely exchange these two properties.

The following result shows that the computation of the associated radial weight $\widetilde{\we}$ on $\Cbb$
reduces to that of $P_\we$, up to a multiplicative constant.

\begin{lemma}\label{l_compute}
Let $\we: [0, +\infty) \to (0, +\infty)$ be a rapid weight function.
Then $6 P_\we(t) \ge \widetilde{\we}(t) \ge P_\we(t)$, $t\ge 0$.
\end{lemma}
\begin{proof}
  First, the definitions of $\widetilde{\we}$ and $P_\we$ guarantee that $\widetilde{\we}(t)\ge P_\we(t)$, $t\ge 0$.

  Secondly, fix a point $\za\in\Cbb$.
By the definition of $\widetilde{\we}(\za)$, there exists a function
$f=f_\za \in\hol(\Cbb)$ such that $|f(z)|\le \we(z)$ for all $z\in\Cbb$
and $\widetilde{\we}(\za) \le 2 |f(\za)|$.
Lemma~III from \cite{EK56} guarantees that $M_f(t) \le 3 P_{M_f}(t)$, $t\ge 0$.
Since $|f(z)|\le \we(z)$ for all $z\in\Cbb$, we have $P_{M_f}(t) \le P_\we (t)$, $t\ge 0$.
Combining the above estimates, we obtain
\[
6 P_\we(|\za|) \ge 6 P_{M_f}(|\za|) \ge 2 |f(\za)| \ge \widetilde{\we}(|\za|).
\]
The point $\za\in\Cbb$ is arbitrary, therefore, $6 P_\we(t) \ge \widetilde{\we}(t)$ for all $t\ge 0$, as required.
\end{proof}

Theorem~\ref{t_abstract} without
property $\we\in\Cbb_{\app}$ is a corollary of Lemma~\ref{l_compute} and certain known results
obtained by different authors in the studies related to the properties $\we\in\Cbb_{\max}$,
$\we\in\Cbb_{\pspc}$ and $\we\in\Cbb_{\mathrm{mon}}$.
Using Lemma~\ref{l_compute}, we add the property $\we\in\Cbb_{\mathrm{ess}}$ to the list of equivalent conditions.

\begin{proof}[of abridged Theorem~\ref{t_abstract}]
It is proved in \cite{Sch95} that the property $\we\in\Cbb_{\mathrm{mon}}$
implies $\we\in\Cbb_{\pspc}$; see also \cite{Ki84} or \cite[Theorem~1]{Sch98}.

Clearly, $\we\in\Cbb_{\pspc}$ implies $\we\in\Cbb_{\max}$;
also, $\we\in\Cbb_{\mathrm{ess}}$ follows from $\we\in\Cbb_{\max}$.

Finally, Lemma~\ref{l_compute} guarantees that
$\we\in\Cbb_{\mathrm{ess}}$ implies $\we\in\Cbb_{\mathrm{mon}}$, that is, $\we\in\Cbb_{\trop}$.
\end{proof}

So, all properties listed in Theorem~\ref{t_abstract}, except $\we\in\Cbb_{\app}$, are equivalent.
It is easy to see that the property $\we\in\Cbb_{\app}$ implies $\we\in\Cbb_{\mathrm{ess}}$.
Hence, to finish the proof of Theorem~\ref{t_abstract},
it suffices to show that $\we\in\Cbb_{\trop}$ implies $\we\in\Cbb_{\app}$.
This implication is obtained in Section~\ref{s_induction}.

\subsection{An analog of Theorem~\ref{t_abstract} for $\dom=\Dbb$}
Theorem~\ref{t_blms} indicates that, for $\dom=\Dbb$, the properties mentioned
in Theorem~\ref{t_abstract} are directly related to the condition $\we\in\Dbb_{\log}$.
In fact, the following analog of Theorem~\ref{t_abstract}
for $\dom=\Dbb$ is a corollary of known results.

\begin{theorem}\label{t_disk}
Let $\we: [0, 1) \to (0, +\infty)$ be an arbitrary weight function.
Then the following properties are equivalent:
\begin{itemize}
  \item $\we\in\Dbb_{\app};$
  \item $\we\in\Dbb_{\mathrm{ess}};$
  \item $\we\in\Dbb_{\max};$
  \item $\we\in\Dbb_{\pspc};$
  \item $\we\in\Dbb_{\trop};$
  \item $\we\in\Dbb_{\log}.$
\end{itemize}
\end{theorem}
\begin{proof}
  Each property under consideration implies that $\we\in\Dbb_{\log}$.
So, to prove the reverse implications, assume that $\we\in\Dbb_{\log}$.

The property $\we\in\Dbb_{\app}$ holds by Theorem~\ref{t_blms}, hence,
$\we\in\Dbb_{\mathrm{ess}}$ also.
Next, it is proved  in \cite{BDL99} that $\we\in\Dbb_{\mathrm{mon}}$,
that is, $\we\in\Dbb_{\trop}$; see also \cite{AD15}.
  We have $\we\in\Dbb_{\pspc}$ by \cite[Lemma~2.2]{Dou14}.
  Finally, $\we\in\Dbb_{\pspc}$ trivially implies $\we\in\Dbb_{\max}$.
\end{proof}

\section{Proof of Theorem~\ref{t_abstract}:
$\we\in\Cbb_{\trop} \Rightarrow \we\in\Cbb_{\app}$}\label{s_induction}
\subsection{Setting}\label{ss_set}
We are given a weight function $\we\in\Cbb_{\trop}$.
Without loss of generality, we assume that $\we$ is log-tropical,
that is, the logarithmic transform $\FF_\we$ is a tropical power series
\[
T(x) = \max_{m\in \Nbb} L_m(x), \quad -\infty \le x < +\infty,
\]
where $L_m(x)= B_m + N_m x$, $B_m\in\Rbb$, $N_m\in\Zbb_+$, $0=N_1< N_2<\dots$.
Initially, $T$ is defined as a supremum. However, if the supremum is finite at each point $x$, $-\infty \le x < +\infty$,
then the supremum realizes as a maximum. Also, without loss of generality,
we assume that each $L_m(x)$ is essential in the definition of $T(x)$, that is,
$T(x)= L_m(x)$ on a non-empty interval for every $m\in \Nbb$.

\subsection{Construction of a thinned tropical series}\label{ss_basind}
Fix a parameter $\hg>0$.
Let $x_{0} = -\infty$ and $\ell_1(x) = L_1(x) = B_1$, $-\infty \le x <+\infty$.

By induction, we construct a subsequence $\{\ell_{k}\}_{k\ge 1}$ of $\{L_{m}\}_{m\ge 1}$
and numbers $x_{k-1}< x_k^\prime \le x_k$, $k=1,2,\dots$.
So, assume that $k\in\Nbb$ and a linear function $\ell_k$ is selected, where $\ell_k = L_m$ for certain $m\ge k$.
Moreover, we are given a point $x_{k-1}$ such that $x_{k-1} < x_k^\prime$, where $x_k^\prime$ is defined as the $x$-coordinate of
the intersection point of $\ell_k = L_m$ and $L_{m+1}$.
Clearly, $T(x_k^\prime) = \ell_k(x_k^\prime)$.
Now, let $s\in\Nbb$ denote the largest number with the following property:

The intersection point of $\ell_k$ and $L_{m+s}$ is strictly above the graph of $T-h$.

Observe that $L_{m+1}$ trivially has the property in question, so the required $s\in\Nbb$ exists.
Put $\ell_{k+1} = L_{m+s}$ and denote by $x_{k}$ the $x$-coordinate of the
intersection point of $\ell_k$ and $\ell_{k+1}$.
So, we have $\xx_k \le x_k < \xx_{k+1}$ and, by the definition of $s\in\Nbb$, the intersection point
of $\ell_k$ and $L_{m+s+1}$ is below the graph of $T-h$.

The functions $\ell_k$ are linear, so we have $\ell_k(x)=b_k + n_k x$ with $b_k\in\Rbb$, $n_k\in\Zbb_+$,
$0=n_1<n_2<\dots$.
Also, put
$t_k = \exp(x_k)$, $k=0,1,\dots$. Clearly, the positive numbers
$t_k$ monotonically increase to $+\infty$ as $k\to \infty$.

Formally, the above construction works for any $\hg>0$.
In applications, we impose additional restrictions, say, $h\ge 4$.

\subsection{Auxiliary lemmas}
\begin{lemma}\label{l_real_picture}
Let the linear functions $\ell_{k}$ and the numbers $x_k$,
$k=1,2,\dots$,
be those constructed in subsection~\ref{ss_basind}.
Then
\begin{align}
\ell_{k-1}(x) \ge \ell_{k+2}(x) +\hg \quad&\textrm{for\ } x\le x_{k-1},\ k\ge 2;
\label{e_picture_1prime}
\\
\ell_{k+2}(x) \ge \ell_{k-1}(x) +\hg \quad&\textrm{for\ } x_{k+1}\le x,\ k\ge 2.
\label{e_picture_2prime}
\end{align}
\end{lemma}
\begin{proof}
We verify property~\eqref{e_picture_1prime}.
The proof of \eqref{e_picture_2prime} is analogous, so we omit it.

Let $(\xi_k, y_k)$ denote the intersection point of $\ell_k$
and $\ell_{k+2}$.
As mentioned in subsection~\ref{ss_basind}, the definition of $s=s(k)\in\Nbb$
guarantees that the intersection point of $\ell_k$ and $L_{m+s+1}$ is below
the graph of $T-h$; hence, $(\xi_k, y_k)$ is also below the graph of $T-h$, that is,
\begin{equation}\label{e_picture_below}
y_k \le T(\xi_k) - h.
\end{equation}

Recall that $x_{k-1}< x_k^\prime \le x_{k}$ and $T(x_k^\prime) = \ell_k(x_k^\prime)$.
Also, observe that $x_{k} < \xi_k$.
For $x\le \xi_k$, the slopes of the linear functions defining $T(x)$ are smaller than
the slope of $\ell_{k+2}$, therefore
\[
\ell_k(\xx_{k}) - \ell_{k+2}(\xx_{k}) = T(\xx_k) - \ell_{k+2}(\xx_k) \ge T(\xi_k) - y_k \ge h
\]
by \eqref{e_picture_below}.

We have $x_{k-1}<\xx_{k}$
and the slope of $\ell_{k+2}$ is larger than that of $\ell_k$,
thus,
\[
\ell_k(x) - \ell_{k+2}(x)\ge \ell_k(\xx_{k}) - \ell_{k+2}(\xx_{k}) \ge h
\]
for $x \le x_{k-1}$.
To finish the proof of \eqref{e_picture_1prime}, it suffices to
observe that $\ell_{k-1}(x) \ge \ell_k(x)$ for $x\le x_{k-1}$.
\end{proof}

\begin{lemma}\label{l_real_segment}
Let the linear functions $\ell_{k}(x)= b_k + n_k x$ and the numbers $t_k= \exp(x_k)$,
$k=1,2,\dots$,
be those constructed in subsection~\ref{ss_basind}.
Put $a_k =\exp(b_k)$ and assume that $\hg\ge 4$.
Then, for $k=1,2,\dots$,
\begin{align}
a_k t^{n_k}
&\le \we(t), \quad t\in [0, +\infty);
\label{e_iLem}
\\
e^{-h} \ff(t)
&\le a_k t^{n_k}, \quad t\in [t_{k-1}, t_k];
\label{e_iiLem}
\\
\sum_{m\ge 1,\ |m-k|\ge 3} a_m t^{n_m}
&\le \frac{1}{2} a_k t^{n_k}, \quad t\in [t_{k-1}, t_k],
\label{e_iiiLem}
\end{align}
where $t_{0}=0$.
\end{lemma}
\begin{proof}
Let $k\ge 1$.
On the one hand, we have $\ell_k(x) \le T(x)= \FF_{\we}(x)$ for $-\infty \le x<+\infty$.
Hence, taking the exponentials, we obtain \eqref{e_iLem}
by the definitions of $\FF_{\we}(x)$ and $\ell_k(x) = b_k + n_k x$.
On the other hand, the inequality
$T(x)-\hg \le \ell_k(x)$, $x\in [x_{k-1}, x_k]$,
implies estimate \eqref{e_iiLem}.

Now, we prove \eqref{e_iiiLem}. 
First, fix a $k\ge 1$ and assume that
$m-k= 3, 6, 9,\dots$. We have
\[
\ell_{k}(x) - \ell_m(x)
=\sum_{j=1}^{(m-k)/3} [\ell_{k+3j-3}(x) -\ell_{k+3j}(x)].
\]
Property~\eqref{e_picture_1prime} guarantees that
$\ell_{k+3j-3}(x) - \ell_{k+3j}(x) \ge \hg$
for $1\le j\le (m-k)/3$ and $-\infty < x\le x_{k+3j-3}$,
hence, for all $x\le x_k$.
In sum,
\[
\ell_k(x) - \ell_m(x) \ge  \hg(m-k)/3
\quad\textrm{for\ } -\infty < x\le x_k.
\]
Since $\ell_{m+2}(x) \le \ell_{m+1}(x) \le \ell_{m}(x)$ for $x\le x_k$,
we also have
\[
\ell_k(x) - \ell_{m+\shift}(x) \ge  \hg(m-k)/3
\quad\textrm{for\ } -\infty < x\le x_k,\ \shift=0,1,2.
\]
Taking the exponentials and using the definitions of $\ell_{m+\shift}$ and $\ell_k$,
we obtain
\[
a_m t^{n_{m+\shift}} \le a_k t^{n_k} \frac{1}{e^{\hg(m-k)/3}} \quad\textrm{for\ } 0\le t\le t_k,\ \shift=0,1,2.
\]
Thus,
\begin{equation}\label{e_m_large}
\sum_{m\ge k+3} a_m t^{n_m}
\le 3 a_k t^{n_k}
\left({e^{-\hg}} + {e^{-2\hg}} + {e^{-3\hg}} + \dots\right)
 \quad\textrm{for\ } 0\le t\le t_k.
\end{equation}

Second, fix a $k\ge 4$.
Replacing \eqref{e_picture_1prime} by \eqref{e_picture_2prime}
and arguing as above, we obtain
\begin{equation}\label{e_m_small}
\sum_{m=1}^{k-3} a_m t^{n_m} \le 3 a_k t^{n_k}
\left({e^{-\hg}} + {e^{-2\hg}} + {e^{-3\hg}} + \dots\right)
\quad\textrm{for\ } t\ge t_{k-1}.
\end{equation}

Since $\hg\ge 4$,
\eqref{e_m_large} and \eqref{e_m_small}
imply \eqref{e_iiiLem}. 
The proof of the lemma is finished.
\end{proof}

\subsection{Proof of the implication $\we\in \Cbb_{\trop} \Rightarrow \we\in\Cbb_{\app}$}\label{ss_abs_prf}
Fix a parameter $\hg\ge 4$.
The induction construction described in subsection~\ref{ss_basind}
provides numbers $n_{k}$, $a_k=\exp(b_k)$ and $t_k=\exp(x_k)$,
$k=1,2,\dots$. Put
\begin{equation}
\GG_\Delta(z) = \sum_{j= 0}^\infty a_{3j+\Delta} z^{n_{3j+\Delta}}, \quad \Delta=1,2,3.
\label{e_GG_def}
\end{equation}
The estimates below guarantee that $\GG_1$, $\GG_2$ and $\GG_3$ are well-defined holomorphic functions of $z\in\Cbb$.
In fact, we claim that
\begin{equation}\label{e_GG}
\frac{1}{2}e^{-h} \we(z) \le |\GG_1(z)|+ |\GG_2(z)|+ |\GG_3(z)| \le 6\we(z),
\quad z\in\Cbb.
\end{equation}

Indeed, assume that $|z| = t\in [t_{3k+\Delta-1}, t_{3k+\Delta}]$ for some $k\ge 0$ and $\Delta\in\{1,2,3\}$.
On the one hand,
\begin{align*}
\frac{1}{2}e^{-h} \we(t)
&\overset{\eqref{e_iiLem}}{\le}
\frac{1}{2}a_{3k+\Delta} t^{\ee_{3k+\Delta}}
\\
&\overset{\eqref{e_iiiLem}}{\le}
a_{3k+\Delta} t^{\ee_{3k+\Delta}} - \sum_{m\ge 1,\ |m-3k-\Delta|\ge 3} a_{m} t^{\ee_{m}}
\\
&\le \left| \sum_{j=0}^\infty a_{3j+\Delta} z^{\ee_{3j+\Delta}} \right| = |\GG_\Delta(z)|
\\
&\le |\GG_1(z)| + |\GG_2(z)| + |\GG_3(z)|.
\end{align*}
On the other hand, if $k\ge 1$, then
\begin{align*}
|\GG_1(z)| + |\GG_2(z)| + |\GG_3(z)|
&\le \sum_{m\ge 1} a_m t^{\ee_m}
\\
&= a_{3k+\Delta} t^{\ee_{3k+\Delta}} + \sum_{m\ge 1,\ |m-3k-\Delta|\ge 3} a_m t^{\ee_m}
\\
&\qquad
+ \left(a_{3k+\Delta-2} t^{\ee_{3k+\Delta-2}} +a_{3k+\Delta-1} t^{\ee_{3k+\Delta-1}} \right.
\\
&\qquad
+ \left. a_{3k+\Delta+1} t^{\ee_{3k+\Delta+1}}+ a_{3k+\Delta+2} t^{\ee_{3k+\Delta+2}}\right)
\\
&\overset{(\ref{e_iLem},\, \ref{e_iiiLem})}{\le}
\left(1+ \frac{1}{2}\right) a_{3k+\Delta} t^{\ee_{3k+\Delta}} + 4\we(t)
\\
&\overset{\eqref{e_iLem}}{\le}6\we(t)
\end{align*}
for $|z| = t\in [t_{3k+\Delta-1}, t_{3k+\Delta}]$.
If $k=0$, then the above estimates are even more simple.
So, \eqref{e_GG} holds.

The proof of Theorem~\ref{t_abstract} is finished.

\section{Explicit conditions}\label{s_explicit}
In this section, we prove Theorem~\ref{t_explicit}.
Methods of \cite{Dom59} are applied in \cite[Lemma~1]{Bo98}
to show that the restriction on $\FF_\we^{\prime\prime}$ formulated in Theorem~\ref{t_explicit}(i) is sufficient
for the property $\we\in\Cbb_{\pspc}$, hence,
for all equivalent properties listed in Theorem~\ref{t_abstract}.
Also, such sufficient conditions are deducible from results of Schmid \cite{Sch95}, \cite{Sch98}.
For the sake of completeness, we give a direct proof of Theorem~\ref{t_explicit}(i).

\begin{proof}[of Theorem~\ref{t_explicit}(i)]
By assumption, there exist $\al>0$ and $A\in\Rbb$ such that
$\FF^{\prime\prime}(x) := \FF^{\prime\prime}_\we(x) \ge \al$ for all $x\ge A$.

Set $L_0(x)\equiv \log\we(0)$, $x\in\Rbb$.
For $n\in\Nbb$, let $L_n$ denote the tangent to the graph of $\FF$ with slope $n$
and at a point $(a_n, y_n)$.
Let $(d_n, z_n)$ denote the intersection point of $L_n$ and $L_{n+1}$.
Put
\[
h_n = \FF(d_n) - L_n(d_n).
\]
Observe that $\we\in \Cbb_{\trop}$ if and only if
\begin{equation}\label{e_h_n}
\sup_{n=0,1,\dots} h_n < \infty.
\end{equation}

Fix $N(A)\in\Nbb$ such that $a_{N(A)} \ge A$.
We claim that $h_n\le 1/\alpha$ for $n\ge N(A)$.
Indeed, we have
\[
1> \FF^\prime(d_n) - \FF^\prime(a_n) =\int_{a_n}^{d_n} \FF^{\prime\prime}(x)\, dx
\ge \alpha(d_n - a_n).
\]
It remains to observe that $d_n - a_n \ge h_n$ by the convexity of $\FF$.
So, property~\eqref{e_h_n} holds, that is, $\we\in \Cbb_{\trop}$.
\end{proof}

\begin{proof}[of Theorem~\ref{t_explicit}(ii)]
Fix an $\er > 0$. Choose $n\in\Nbb$ such that $\FF^{\prime\prime}(x) < \er$ for all $x\in [n, +\infty)$.
Recall that $\we$ is rapid, so $\FF^\prime$ increases from zero to $+\infty$.
Hence, there exist numbers $a_n, b_n\in\Rbb$ such that
$\FF^\prime(a_n)=n$ and $\FF^\prime(b_n)=n + 1/2$.
We have
\begin{align*}
\FF(b_n)
&=\FF(a_n) + \int_{a_n}^{b_n}\FF^\prime(x)\, dx,
\\
L_n(b_n)
&= \FF(a_n) + n (b_n-a_n),
\end{align*}
where 
$L_n$ denotes the tangent to the graph of $\FF$ with slope $n$.

Let $\psi$ denote the function inverse to $\FF^\prime$.
Observe that $\psi^\prime(u) \ge \frac{1}{\er}$
for all $u\in [n, n+1]$.
Therefore,
\begin{equation}\label{e_fiLn}
  \begin{split}
  \Phi(b_n) - L_n(b_n) &= \int_{a_n}^{b_n} \left(\FF^\prime(x)-n\right)\, dx
  \\
  &= \int_n^{n+\frac{1}{2}} \left( \psi\left(n + \frac{1}{2} \right)-\psi(u) \right)\, du
  \\
  &\ge \int_n^{n+\frac{1}{2}} \frac{1}{\er} \left(n+\frac{1}{2} -u \right)\, du
  \\
  &=   \frac 1{8\er}
 \end{split}
\end{equation}
Analogously, we obtain
\begin{equation}\label{e_fiLn1}
  \Phi(b_n) - L_{n+1}(b_n) \ge  \frac 1{8\er}.
\end{equation}
Using \eqref{e_fiLn} and \eqref{e_fiLn1}, it is easy to conclude that
$h_n \ge \frac 1{8\er}$.
Since $\er>0$ is arbitrarily small,
property~\eqref{e_h_n} does not hold,
or, equivalently, $\we\not\in\Cbb_{\trop}$.
\end{proof}

\begin{remark}
Let $\we: [0, +\infty) \to (0, +\infty)$ be a log-convex and ${\mathcal C}^2$-smooth
rapid weight function.
Simple examples show that both properties $\we\in\Cbb_{\trop}$ and $\we\notin\Cbb_{\trop}$
are compatible with oscillation of $\FF^\prime$, that is, with conditions
$\limsup_{x\to +\infty}\FF_\we^{\prime\prime}(x)>0$ and
$\liminf_{x\to +\infty}\FF_\we^{\prime\prime}(x)=0$.
\end{remark}

\section{Further results}\label{s_further}

\subsection{Holomorphic approximation in $\cd$}\label{ss_cd}
Let $d\ge 1$. Given a weight function $\we:[0,+\infty) \to (0,+\infty)$, we put
$\we(z) = \we(|z|)$, $z\in\cd$, to obtain the corresponding radial weight on $\cd$.
Definitions~\ref{def_n_app}--\ref{d_max} are naturally applicable to $\we$,
so the properties $\we\in\cd_{\app}$, $\we\in\cd_{\mathrm{ess}}$ and $\we\in\cd_{\max}$ are defined.
In this section, we extend Theorem~\ref{t_abstract} to $\dom=\cd$, $d\ge 1$.

\begin{theorem}\label{t_abstract_d}
Let $\we: [0, +\infty) \to (0, +\infty)$ be a weight function.
Then the following properties are equivalent:
\begin{itemize}
  \item the radial weight $\we$ on $\cd$ is approximable by a holomorphic map for all (some) $d\ge 1$ $(\we\in\cd_{\app});$
  \item $\we$ is approximable by the maximum of a holomorphic function modulus on $\cd$ for all (some) $d\ge 1$ $(\we\in\cd_{\max});$
  \item $\we$ is essential on $\cd$  for all (some) $d\ge 1$ $(\we\in\cd_{\mathrm{ess}});$
  \item $\we$ is equivalent to a log-tropical weight function $(\we\in\Cbb_{\trop})$.
\end{itemize}
\end{theorem}

\begin{proof}[About the proof of Theorem~\ref{t_abstract_d}]
As in Theorem~\ref{t_abstract}, the situation essentially simplifies when $\we$ is not rapid:
$\cd$-properties in question hold if and only if $\we(t)\asymp 1+t^m$, $t\ge 0$, for certain $m\in\Nbb$;
see Remark~\ref{r_nonrapid_2}.
So, without loss of generality, we may assume that $\we$ is a rapid weight function.

Standard arguments guarantee that $\we\in\Cbb_{\mathrm{ess}}$
if and only if $\we\in\cd_{\mathrm{ess}}$, $d\ge 2$.
Moreover, we clearly have the following implications:
\[
\we\in\Cbb_{\max} \Rightarrow \we\in\cd_{\max} \Rightarrow \we\in\cd_{\mathrm{ess}}.
\]
Therefore, the properties $\we\in\cd_{\mathrm{ess}}$ and $\we\in\cd_{\max}$ do not depend on $d\ge 1$,
and they are equivalent to all properties listed in Theorem~\ref{t_abstract}, in particular,
to the property $\we\in\Cbb_{\trop}$.

If $\we\in\cd_\app$, then clearly $\we\in\Cbb_\app$.
Hence, $\we\in\Cbb_{\trop}$ by Theorem~\ref{t_abstract}.
So, to prove the theorem, it suffices to show that the property
$\we\in \Cbb_{\trop}$ implies $\we\in\cd_{\app}$.
For $d=1$, Theorem~\ref{t_abstract} applies.
For $d\ge 2$, we argue as in the proof of the implication
$\we\in \Cbb_{\trop} \Rightarrow \we\in\Cbb_{\app}$, replacing the monomials
$z^k$, $k\ge 0$, by Aleksandrov--Ryll-Wojtazsczyk polynomials of $d$ variables
(see \cite{Aab86}, \cite{RW83}).
Namely, there exist $Q=Q(d)\in\Nbb$ and $\delta=\delta(d)>0$ such that the following property holds:
for $k=0,1,2,\dots$, there exist homogeneous holomorphic polynomials $W_q[k]$ of degree $k$, $q=1,2,\dots, Q$,
such that
\begin{equation}\label{e_RW}
|W_q[k](\za)|\le 1\quad\textrm{and}\quad\max_{q=1,2,\dots, Q} |W_q[k](\za)|\ge\delta
\quad \textrm{for all}\ \za\in\cd,\ |\za|=1.
\end{equation}
Put
\[
g_{\Delta, q}(z) = \sum_{j=0}^\infty a_{3j+\Delta} W_q[n_{3j+\Delta}](z),
\quad \Delta=1,2,3,\ q=1,2,\dots, Q,
\]
where the induction construction from subsection~\ref{ss_basind} is used
to define the numbers $n_{k}$, $a_k=\exp(b_k)$ and $t_k=\exp(x_k)$,
$k=1,2,\dots$.
Also, applying \eqref{e_m_large} and \eqref{e_m_small}, select $h=h(\delta)>4$ so large that
\begin{equation}\label{e_delta}
\sum_{m\ge 1,\ |m-k|\ge 3} a_m t^m
\le \frac{\delta}{2} a_k t^k, \quad t\in [t_{k-1}, t_k],\ k=1,2,\dots.
\end{equation}
Now, we claim that
\[
\frac{\delta}{2} e^{-h} \we(z) \le
\sum_{\Delta=1}^3 \sum_{q=1}^Q |\GG_{\Delta, q}(z)| \le 6Q\we(z), \quad z\in\cd.
\]
To prove the above estimates, we use \eqref{e_RW} and \eqref{e_delta}.
The argument is essentially the same as in subsection~\ref{ss_abs_prf},
so we omit further details.
\end{proof}

\subsection{Approximation by harmonic maps}\label{ss_harm}

In this section, we consider a harmonic analog of the key property $\we\in\cd_{\app}$.
In particular, we obtain a necessary condition for harmonic approximation in $\RRq$, $q\ge 2$.
If the dimension $q$ is even, then the condition under consideration is also necessary.

\begin{definition}\label{def_harm}
A radial weight $\we$ on $\RRq$, $q\ge 2$, is called \textsl{approximable by a harmonic map}
(in brief, $\we\in\RRq_{\har}$) if there exists $m\in\Nbb$ and a harmonic map $h: \RRq\to \Rbb^m$ such that
\[
|h(y)| \asymp \we(y), \quad y\in\RRq.
\]
\end{definition}

\begin{lemma}\label{l_harm}
Let $\we: [0,+\infty) \to (0, +\infty)$ be a weight function.
Then the following properties are equivalent:
\begin{itemize}
\item[(a)] $\we$ is equivalent to a log-tropical weight function $(\we\in\Cbb_{\trop});$
\item[(b)] $\we^2(t)$ is equivalent to a power series
          $s_2(t) = \sum_{k=0}^\infty a_k t^{2k}$ with $a_k\ge 0$, $t\ge 0$.
\end{itemize}
\end{lemma}
\begin{proof}
(a)$\Rightarrow$(b)
We have $\we\in\Cbb_{\trop}$, that is, $\we(t)$ is equivalent to $v(t)$
such that $\FF_v(x) = \max_{k\ge 0} (b_k + kx)$.
Therefore, $\we^2(t) \asymp v^2(t) = F(t)$, where
$\FF_F(x)= \max_{k\ge 0} (2b_k + 2kx)$.
Observe that $F(t)= \max_{k\ge 0} e^{2b_k} t^{2k}$, hence, $F(t)= G(t^2)$,
where $\FF_G(x) = \max_{k\ge 0} (2b_k + kx)$.
Since $G(t)$ is a log-tropical weight function, Theorem~\ref{t_abstract}
provides a power series $s(t)= \sum_{k=0}^\infty a_k t^k$ such that $a_k\ge 0$ and $G(t)\asymp s(t)$.
In sum, $\we^2(t) \asymp F(t) = G(t^2) \asymp s(t^2) = \sum_{k=0}^\infty a_k t^{2k}$, $a_k \ge 0$.
So, (a) implies (b).

(b)$\Rightarrow$(a)
We are given a power series $s_2(t) = \sum_{k=0}^\infty a_k t^{2k}$ such that $a_k\ge 0$ and $s_2(t)\asymp w^2(t)$.
Applying Theorem~\ref{t_abstract} to $q(t)= \sum_{k=0}^\infty a_k t^{k} \in\Cbb_{\pspc}$, we obtain a log-tropical weight function $G(t)$
such that $q(t) \asymp G(t)$.
Put $F(t)= G(t^2)$, then $s_2(t) = q(t^2) \asymp G(t^2) = F(t)$.
Also, we have $\FF_F(x) = \max_{k\ge 0} (b_k + 2kx)$ for certain $b_k\in\Rbb$.
The equality $\FF_f(x) = \max_{k\ge 0} (b_k/2 + kx)$
defines a log-tropical weight function $f(t): [0,+\infty) \to (0, +\infty)$
such that $f^2(t) = F(t)$. In sum, we obtain $\we^2(t) \asymp s_2(t) \asymp F(t) = f^2(t)$.
So, $\we(t)\asymp f(t)\in \Cbb_{\trop}$, that is, (b) implies (a).
\end{proof}

\begin{proposition}\label{p_harm}
Let $q\ge 2$ and let $\we$ be a radial weight on $\RRq$.
\begin{itemize}
\item[(i)]
If $\we\in\RRq_{\har}$, then the weight function $\we$ is equivalent to
a log-tropical weight function.
\item[(ii)]
If $q$ is even, then $\we\in\RRq_{\har}$ if and only if $\we\in\Cbb_{\trop}$.
\end{itemize}
\end{proposition}
\begin{proof}
(i) Let $\we\in\RRq_{\har}$. So, we are given a harmonic map $h=(h_1, \dots, h_m): \RRq\to \Rbb^m$
such that
\[
|h_1(y)|^2 + \dots + |h_m(y)|^2 \asymp \we^2(y), \quad y\in\RRq,
\]
or, equivalently,
\begin{equation}\label{e_har_equiv}
|h_1(t\xi)|^2 + \dots + |h_m(t\xi)|^2 \asymp \we^2(t), \quad 0\le t <+\infty,\ \xi\in S_q,
\end{equation}
where $S_q$ denotes the unit sphere of $\RRq$.

For $j=1,\dots, m$, we have
\[
h_j(y) = \sum_{k=0}^\infty P_{j,k}(y), \quad y\in \RRq,
\]
where $P_{j,k}$ is a harmonic homogeneous polynomial of degree $k$, and
the series converges uniformly on compact subsets of $\RRq$.
Let $\si_q$ denote the normalized Lebesgue measure on $S_q$.
For $k_1\neq k_2$, $P_{j, k_1}$ and $P_{j, k_2}$ are orthogonal
in $L^2(S_q, \si_q)$, thus
\[
\int_{S_q} |h_j(t\xi)|^2 \, d\si_q(\xi)
=
\sum_{k=0}^\infty t^{2k} \|P_{j,k}\|_{L^2(S_q, \si_q)}^2,
\quad 0\le t < +\infty,\ j=1,2,\dots, m.
\]
Therefore, integrating \eqref{e_har_equiv} with respect to $\si_q$,
we obtain
\[
\we^2(t) \asymp \sum_{k=0}^\infty a_k^2 t^{2k}\quad \textrm{for certain\ } a_k\in\Rbb.
\]
Applying Lemma~\ref{l_harm}, we deduce that $\we\in\Cbb_{\trop}$.

(ii) Let $q=2d$.
If $\we\in\Cbb_{\trop}$, then Theorem~\ref{t_abstract_d} provides a holomorphic map
$f=(f_1,\dots, f_n): \cd\to \Cbb^n$ such that $|f|\asymp \we$.
Identifying $\RRq$ and $\cd$ and putting $h=(\RRe f_1, \IIm f_1, \dots, \RRe f_n, \IIm f_n)$,
we conclude that $\we\in\Cbb_{\trop}$ implies $\we\in\RRq_{\har}$.
The reverse implication holds by part~(i) of the proposition.
\end{proof}

\section{Applications}\label{s_applic}

\subsection{Proper holomorphic immersions and embeddings}\label{ss_proper}

A proper holomorphic map is \textsl{an immersion} if
its Jacobian is non-degenerate everywhere.
By definition, \textsl{a proper holomorphic embedding} is a proper holomorphic immersion which is one-to-one.

\subsubsection{Embeddings of the unit disk}
A proper holomorphic embedding $f: \Dbb\to\Cbb^2$ is not trivial to construct;
the first example was given by K.~Kasahara and T.~Nishino (see \cite{KN1969}, \cite{Steh70}).
Later the unit disk was replaced by an annulus (see \cite{Lau73}), the punctured disk (see \cite{AlexH77})
and more sophisticated planar domains.
However, the following problem remains open (see \cite{FW13}, \cite{GS95}):
Can any planar domain be properly holomorphically
embedded to $\Cbb^{2}$?

As mentioned in Section~\ref{ss_app_proper}, J.~Globevnik \cite{G02Edin}
asked whether a proper holomorphic embedding $f: \Dbb\to\Cbb^2$ may grow arbitrarily slowly.
Using Theorem~1.2 from \cite{AD15}, we obtain the following result.

\begin{corollary}[see {\cite[Corollaries~2.3 and 2.4]{ADcras}}]\label{c_disk_imm_emb}
Let $\we$ be a log-convex radial weight on $\Dbb$.
Then there exists a proper holomorphic immersion $f: \Dbb\to\Cbb^2$ such that $|f|\asymp \we$.
Also, there exists a proper holomorphic embedding $f: \Dbb\to\Cbb^3$ such that $|f|\asymp \we$.
\end{corollary}

\begin{remark}\label{r_globevnik}
For the radial weights on $\Dbb$,
the log-convexity is a regularity condition, not a growth one.
In particular, a log-convex weight function may grow arbitrarily slowly or arbitrarily rapidly.
So, a proper holomorphic immersion $f: \Dbb\to\Cbb^2$ may grow arbitrarily slowly.
It would be interesting to know whether Corollary~\ref{c_disk_imm_emb}
extends to appropriate holomorphic embeddings $f: \Dbb\to\Cbb^2$.
\end{remark}

\subsubsection{Embeddings of the unit ball}
For the unit ball $\bd$ of $\cd$, $d\ge 2$, proper holomorphic embeddings $f: \bd\to\Cbb^n$
have been investigated by many authors
in a more general setting of Stein manifolds $M_d$ of dimension $d$.
The following essentially sharp result
was obtained by Y.~Eliashberg and M.~Gromov \cite{EG92}:
Every Stein manifold $M_d$ of dimension $d$ can be properly
holomorphically embedded to $\Cbb^{n(d)}$
for the minimal integer $n(d) > (3d+1)/2$.

For $\bd$ with arbitrary $d\ge 1$, Theorem~1.3 from \cite{AD15} implies the following assertion.

\begin{corollary}\label{c_ball_emb}
Let $\we$ be a log-convex radial weight on $\bd$. Then there exists a number $n=n(d)$
and a proper holomorphic embedding $f: \bd\to\Cbb^{n(d)}$ such that $|f|\asymp \we$.
\end{corollary}

\subsubsection{Embeddings of $\cd$, $d\ge 1$}
If $\we: [0,+\infty) \to (0, +\infty)$ is an arbitrary rapid log-convex
radial weight, then clearly there is no direct extension of Corollary~\ref{c_disk_imm_emb}
to proper holomorphic immersions and embeddings of $\Cbb$ into appropriate $\Cbb^n$.
Indeed, by Theorem~\ref{t_abstract}, $\we$ has to be equivalent to a log-tropical weight function.
In fact, Theorem~\ref{t_abstract} provides the following result of this type.

\begin{corollary}\label{c_cbb_imm}
Let $\we$ be a log-tropical rapid radial weight on $\Cbb$.
Then there exists a proper holomorphic immersion $f: \Cbb\to\Cbb^3$ such that $|f|\asymp \we$.
\end{corollary}
\begin{proof}
Given a log-tropical rapid radial weight $\we$, formula~\eqref{e_GG_def}
defines a holomorphic map $g=(g_1, g_2, g_3): \Cbb\to \Cbb^3$ such that $|g| \asymp \we$,
$g_1(0) = \we(0)>0$
and $g_3^\prime(0)=0$.
Put $h_1(z) = g_1(z)$, $h_2(z)=g_2(z)$ and $h_3(z) = z + g_3(z)$.
Then $h_1(0)\neq 0$ and $h_3^\prime(0) \neq 0$.
The zero sets of $h_j(z)$ and $h_j^\prime(z)$, $j=1,3$, are countable ones,
hence, there exists $\theta\in [0, 2\pi)$ such that,
for $f_j(z)= h_j(z)$, $j=1,2$, and $f_3(z) = h_3(e^{i\theta}z)$, we have
$(f_1(z), f_2(z), f_3(z))\neq (0,0,0)$ and $(f_1^\prime(z), f_2^\prime(z), f_3^\prime(z))\neq (0,0,0)$
for all $z\in\Cbb$.
To finish the proof, it remains to observe that $|f_1|+|f_2|+|f_3| \asymp \we$
by the arguments used in Section~\ref{ss_abs_prf}.
\end{proof}

\begin{corollary}\label{c_cbb_emb}
Let $\we$ be a log-tropical rapid radial weight on $\Cbb$.
Then there exists a proper holomorphic embedding $f: \Cbb\to\Cbb^4$ such that $|f|\asymp \we$.
\end{corollary}
\begin{proof}
Corollary~\ref{c_cbb_imm} provides a proper holomorphic immersion $g=(g_1, g_2, g_3): \Cbb\to\Cbb^3$
such that $|g|\asymp \we$. It remains to set $f(z)=(g_1(z), g_2(z), g_3(z), z)$, $z\in\Cbb$.
\end{proof}

Similarly, Theorem~\ref{t_abstract_d} implies the following result.
\begin{corollary}\label{c_cd_emb}
Let $\we$ be a log-tropical rapid radial weight on $\cd$, $d\ge 1$. Then there exists a number $n=n(d)$
and a proper holomorphic embedding $f: \cd\to\Cbb^{n(d)}$ such that $|f|\asymp \we$.
\end{corollary}

\subsection{Operators on the growth spaces}\label{ss_GrSpaces}
Standard applications of the property $\we\in\cd_{\app}$, $d\ge 1$,
are related to concrete operators on the growth space $\mathcal{A}^{\we}(\cd)$;
in particular, Volterra type operators are considered in \cite{AD17}.
Also, see \cite{AD12}, \cite{AD15} and references therein for various applications
of Theorem~\ref{t_blms} and its generalizations
in the setting of the growth spaces on the unit ball of $\cd$, $d\ge 1$.

\bibliographystyle{amsplain}      

\end{document}